\documentclass[11pt]{amsart}
\usepackage[reset,a4paper,vmargin=3truecm,hmargin=2truecm]{geometry}

\usepackage{amsmath,amssymb,graphicx,bbm,amsthm,graphicx}
\usepackage{mathtools}
\usepackage{leftidx}
\usepackage{hyperref}
\usepackage{color}
\usepackage{cool}

\newtheorem{theorem}{Theorem}[section]

\newtheorem{corollary}[theorem]{Corollary}

\newtheorem{lemma}[theorem]{Lemma}
\newtheorem{proposition}[theorem]{Proposition}

\theoremstyle{definition}
\newtheorem{example}[theorem]{Example}
\theoremstyle{remark}
\newtheorem{remark}[theorem]{Remark}

\makeatletter
\def\contFracOpe{%
    \operatornamewithlimits{%
        \mathchoice{
            \vcenter{\hbox{\huge $\mathcal{K}$}}%
        }{
            \vcenter{\hbox{\Large $\mathcal{K}$}}%
        }{
            \mathrm{\mathcal{K}}%
        }{
            \mathrm{\mathcal{K}}%
        }
    }
}

\newcommand{\subplus}{\mathbin{\genfrac{}{}{0pt}{}{}{+}}}
\newcommand{\subminus}{\mathbin{\genfrac{}{}{0pt}{}{}{-}}}
\newcommand{\subcdots}{\genfrac{}{}{0pt}{}{}{\cdots}}

\begin{document}

\title{A simple continued fraction expansion for $e^n$}

\author{Cid Reyes-Bustos}

\date{\today}

\begin{abstract}
  In this paper we present a family of continued fraction expansions for $e^n$, with $n\ge 1$, with a simple expression having partial
  denominators given by arithmetic progressions. We give an estimate for the convergence speed showing that the convergence is
  faster than the corresponding regular continued fractions.
  Moreover, we prove that the continued fractions for $e^n$ given in this paper are special cases of continued fraction expansions,
  different from the standard ones, of the confluent hypergeometric function,
  or equivalently, of the incomplete gamma function. In addition, using the same method we give a related family of continued
  fraction expansions of $e^{\ell/n}$ for positive integers \(1\leq\ell< n \) that contains the case of integral exponent as a limit case. 
\end{abstract}



\maketitle

\section{Introduction}

There are several well-known continued fraction expansions for the Euler number \(e\), the base of the natural logarithm, and its
integral and rational powers. For instance, the continued fractions
\begin{align*}
  e &= 2 + \frac{1}{1} \subplus \frac{1}{2} \subplus \frac{1}{1} \subplus \frac{1}{1} \subplus \subcdots = \left[2; \overline{1,2n,1}\right]_{n=0}^{\infty}, \\
  e &= 2 + \frac{2}{2} \subplus \frac{3}{3} \subplus \frac{4}{4} \subplus \frac{5}{5} \subplus \subcdots, \\
  e &= 1 + \frac{2}{1} \subplus \frac{1}{6} \subplus \frac{1}{10} \subplus \frac{1}{14} \subplus \subcdots. 
\end{align*}
can be found in Appendix A of \cite{Lorentzen2008}. For the rational powers we have, for example, the continued fraction
\[
  e^2 = 7 + \frac{1}{2} \subplus \frac{1}{1} \subplus \frac{1}{1} \subplus \frac{1}{3} \subplus \subcdots = \left[7;\overline{3n+2,1,1,3n+3,12n+18}\right]_{n=0}^{\infty},
\]
and, for \(M>1\), 
\begin{equation}
  \label{eq:e1m}
  e^{1/M} = 1 + \frac{1}{M-1} \subplus \frac{1}{1} \subplus \frac{1}{1} \subplus \frac{1}{3M-1} \subplus \subcdots,
\end{equation}
along with a related continued fraction for \(e^{2/M} \), obtained in \cite{McCabe2009} by the use of the
quotient-difference algorithm for Padé tables (see \cite{McCabe1983}). 

In this paper, as an starting point we consider the continued fraction
\begin{equation}
  \label{eqn:contfrac1}
  e = 3 - \frac{1}{4} \subminus \frac{2}{5} \subminus \frac{3}{6} \subminus \frac{4}{7} \subminus \subcdots,
\end{equation}
discovered algorithmically in \cite{Raayoni2019} and proved in \cite{Lu2019} (see also \cite{DZ2021} and \cite{RB2018}). By comparison of the initial convergents, it is clear that the continued fraction \eqref{eqn:contfrac1} is not obtained from a equivalence transformation of any of the above continued fractions.

We give an alternate proof of the continued fraction expansion \eqref{eqn:contfrac1} by establishing the formula
\begin{equation}
  \label{eq:contfracen}
    e^n = \sum_{k=0}^{n-1} \frac{n^k}{k!} + \frac{n^{n-1}}{(n-1)!} \left( 1 + n +
    \contFracOpe_{m=1}^{\infty} \left(\frac{-n(m+n-1)}{m+2 n + 1}\right) \right),
\end{equation}
valid for \(n \in \mathbb{N}\). Here, the notation
\[
  \contFracOpe_{m=1}^{\infty} \left(\frac{a_m}{b_m}\right) := \frac{a_1}{b_1} \subplus \frac{a_2}{b_2} \subplus \frac{a_3}{b_3} \subplus \subcdots
\]
is used for the continued fraction with partial numerators \(a_n\) and partial denominators \(b_n\).

The convergents of the continued fractions \eqref{eq:contfracen} have certain interesting properties. We illustrate this with the case \(n=1\). Denote by \( C_k = P_k/Q_k \)  the \(k\)-th convergent of (\ref{eqn:contfrac1}). The first convergents
of \eqref{eqn:contfrac1} are given by
\begin{align*}
  C_0 = 3, \quad C_1 = \frac{11}{4}, \quad  C_2 = \frac{49}{18}, \quad C_3= \frac{87}{32}, \quad C_4 = \frac{1631}{600}, \quad C_5 = \frac{11743}{4320}.
\end{align*}
Note that since the partial numerators of \eqref{eqn:contfrac1} are negative, the convergents are monotonically decreasing. The difference between the first successive convergents is given by
\begin{align*}
 C_1 - C_0 = \frac{-1}{4}, \quad C_2 - C_1 = \frac{-1}{36}, \quad C_3 - C_2 = \frac{-1}{784}, \quad C_4 - C_3 = \frac{-1}{2400}.
\end{align*}
In general (see Proposition \ref{prop:diff}), the absolute value of the numerators of the differences is \(1\), similar to the case of regular continued fractions. 

The convergence of the continued fraction is faster than the regular continued fraction. In fact,
in Corollary \ref{cor:convspeed}, we see that
\[
  \left|e - C_k \right| = O\left(\frac{1}{k! (k+1)^2(k+2)^2 }\right),
\]
for \(k \in \mathbb{N}\). In \S2 we prove these properties by establishing the corresponding ones for the continued fractions \eqref{eq:contfracen}.

The continued fraction \eqref{eq:contfracen} in general cannot be extended to a continued fraction for \(e^z\) where \(z\) is a complex variable (or even a real number different of \(z=n \in \mathbb{N}\)). In fact, in \S 3 we show that the appropriate extension of \eqref{eq:contfracen} to complex variable (Theorem \ref{thm:incgamma}) corresponds to a continued fraction expansion of the incomplete gamma function, or equivalently, of the confluent hypergeometric function, that is different from the standard ones. 


Finally, in \S \ref{sec:extension}, by extending the method used for \eqref{eq:contfracen}, we show the convergence of
the family of continued fractions
\[
  1 + 2n + \contFracOpe_{m=1}^{\infty} \left( \frac{-\, n m}{1+ n (m + 2)} \right) = \frac{e^{\frac{1}{n}}}{n-(n-1)e^{\frac{1}{n}}},
\]
valid for $n>0$. Note that the case $n=1$ gives \eqref{eqn:contfrac1}, and thus this family is another generalization of \eqref{eqn:contfrac1}. In general, we give a related two parameter family of continued fractions for $e^{\ell/n}$, with \(1\leq\ell<n\), that includes \eqref{eq:contfracen} as a limit case.


We remark that the continued fraction \eqref{eqn:contfrac1} was discovered (independently of \cite{Raayoni2019}) by the author in 2018 (see Chapter 5 in \cite{RB2018}) during the study of orthogonal polynomials related to solutions of certain second order differential operator of confluent Heun type related to a model used in quantum optics (the asymmetric quantum Rabi model (AQRM), see e.g. \cite{KRW2017,CRB2020}).
In \cite{DZ2021}, the authors prove the convergence of continued fractions related to \eqref{eqn:contfrac1} by presenting a family of continued fractions equivalent to the well-known M-fraction of the confluent hypergeometric function. In Remark \ref{rem:other} we compare our results with the ones given in \cite{DZ2021}.

\section{The continued  fraction expansion of \( e^n\) }
\label{sec:cont}

In this section we prove the convergence of the continued fraction expansion (\ref{eq:contfracen}) and provide an estimate for its convergence speed. 

\begin{theorem}
  \label{thm:exp_cont}
  For \( n \in \mathbb{N}\) we have
  \[
    e^n = \sum_{k=0}^{n-1} \frac{n^k}{k!} + \frac{n^{n-1}}{(n-1)!} \left( 1 + n +
    \contFracOpe_{m=1}^{\infty} \left(\frac{-n(m+n-1)}{m+2 n + 1}\right) \right)
  \]
\end{theorem}

We establish the convergence by constructing a tail sequence using an associated recurrence relation
and then using the Waadeland tail theorem (see e.g. Chapter 2 of \cite{Lorentzen2008}).
We make use of some lemmas to prove this result.

\begin{lemma}
  \label{lem:recurr}
  For a fixed \(n \in \mathbb{N}\), the recurrence relation
  \[
    X_k = (k + 2n + 2) X_{k-1} - n(k+n) X_{k-2},
  \]
  has the solution
  \[
    X_k = (k+2)\Gamma(k + n + 2),
  \]
  for \(k \geq 1 \).
\end{lemma}

\begin{proof}
  We verify directly. The right-hand side equals
  \begin{align*}
    (k + 2n + 2) &(k+1)\Gamma(k+n+1) - n (k+n)k\Gamma(k+n) \\
                 &= \Gamma(k+n+1) \left( (k+2n + 2)(k+1) - n k) \right) \\
                 &= \Gamma(k+n+1) (k + n + 1) (k+2) \\
                 &= (k+2)\Gamma(k+n+2),
  \end{align*}
  which is equal to \(X_k\), as desired.
\end{proof}
The next lemma is used for the evaluation of the hypergeometric series that appear
in the computation of the limit by Waadeland's Tail Theorem. We recall the definition
of the incomplete gamma function \(\gamma(s,x) \), namely
\[
  \gamma(s,x) := \int_0^x t^{s-1} e^{-t} d t.
\]
Also we note that, for \(p,q\geq 1 \) the notation \({}_pF_q\left(
    \begin{matrix}
      a_1 ,& a_2, & \ldots & a_p \\ b_1 ,& b_2, & \ldots & b_q
    \end{matrix}\, ;\, x \right) \) is used for the generalized hypergeometric series in the standard way (see e.g. \cite{AAR1999}).

\begin{lemma}
  \label{lem:hyper}
  The formula
  \[
    {}_2F_2\left(
    \begin{matrix}
      1 ,& 1 \\ 3 ,& z+2
    \end{matrix}\, ;\, z \right) = \frac{2(z+1)}{z^2} \left( 1 + z - \frac{\gamma(z,z)}{z^{z-1} e^{-z}} \right),
  \]
  is valid for \(z\) in the cut plane \( z \in \{ z \in \mathbb{C} : |\arg(z)| < \pi  \} \).
  In particular, for  \( n \in \mathbb{N}\) we have
  \[
    {}_2F_2\left(
    \begin{matrix}
      1 ,& 1 \\ 3 ,& n+2
    \end{matrix}\, ;\, n \right) = \frac{2(n+1)}{n^2} \left( 1+ n - \frac{(n-1)!}{n^{n-1}}\left(e^n - \sum_{k=0}^{n-1}\frac{n^k}{k!} \right) \right).
  \]
\end{lemma}

\begin{proof}
  Using the Euler's integral transform two times to the hypergeometric series, we get
  \begin{align*}
    {}_2F_2\left(
    \begin{matrix}
      1 ,& 1 \\ 3 ,& z+2
    \end{matrix}\, ;\, z \right)
       &= (z + 1) \int_0^1 (1-t)^z {}_1F_1\left(
    \begin{matrix}
      1 \\ 3
    \end{matrix}\, ;\, t z \right)    d t \\
       &= 2 (z+1) \int_0^1 (1-t)^z \int_0^1 (1-s) e^{s t z} d s d t \\
       &= \frac{2 (z+1)}{z^2} \int_0^1 \frac{(1-t)^z}{t^2} \left( e^{z t} - z t - 1 \right) d t,
  \end{align*}
  then, by partial integration, we obtain
  \begin{equation}
    \label{eq:integral}
    \frac{2 (z+1)}{z}  \int_0^1 (1-t)^{z-1} \left( 1+ z - e^{z t} \right) d t = \frac{2 (z+1)}{z} \left( \frac{1+z}{z} - \int_0^1 (1-t)^{z-1}  e^{z t}  d t \right).
  \end{equation}
  A change of variable \(s = 1-t \) gives the first statement of the lemma
  \begin{equation*}
    \frac{2 (z+1)}{z} \left( \frac{1+z}{z} - \frac{e^{z}}{z^z}\int_0^z s^{z-1}  e^{- s}  d s \right)
    = \frac{2 (z+1)}{z} \left( \frac{1+z}{z} - \frac{e^{z}}{z^z}\gamma(z,z) \right).
  \end{equation*}
  For the second statement, recall the expression for the residue term \(R_n\) of the \(n\)-th
  order Taylor's expansion of \(f(x) = e^x\) around \(x =0\) (see e.g. (5.41) of \cite{Whittaker1950}),
  evaluated at \( n\),
  \[
    R_n = \frac{n^n}{(n-1)!} \int_0^1 (1-t)^{n-1} e^{n t} d t = e^n - \sum_{k=0}^{n-1} \frac{n^k}{k!}.
  \]
  Substituting in (\ref{eq:integral}) gives the desired expression
  \begin{equation*}
    \frac{2(n+1)}{n^2} \left( 1+ n - \frac{(n-1)!}{n^{n-1}}\left(e^n - \sum_{k=0}^{n-1}\frac{n^k}{k!} \right) \right).
  \end{equation*}
\end{proof}

Next, we present the proof of Theorem  \ref{thm:exp_cont}. Recall that for any \(a \in \mathbb{C} \), \((a)_n \) denotes the Pochhammer symbol, that is \((a)_0 := 1\), and
\[
  (a)_n := a (a+1) \cdots (a+n-1),
\]
for \(n \in \mathbb{N}\).

\begin{proof}[Proof of Theorem \ref{thm:exp_cont}]
  We obtain the result by computing the  value of the shifted continued fraction
  \(\contFracOpe_{m=1}^{\infty} (\frac{-n(m+n)}{m+2 n + 2})\). By Lemma \ref{lem:recurr},
  a tail sequence \(\{t_j\}_{j=0}^{\infty}\) for the continued fraction is given by
  \[
    t_j = -\frac{(j+n+1)(j+2)}{j+1},    \qquad t_0 = - 2(n+1).
  \]
  Let \(b_j = j + 2 n + 2 \) for \(j \in \mathbb{N}\) and define
  \[
    \Sigma_l := \sum_{k=0}^l \prod_{j=1}^k \left(\frac{b_j + t_j}{-t_j} \right)  =\sum_{k=0}^l \frac{(1)_k (1)_k}{(3)_k (n+2)_k} \frac{n^k}{k!},
  \]
  then, taking the limit and using Lemma \ref{lem:hyper}, we obtain
  \[
    \Sigma_\infty =     {}_2F_2\left(
    \begin{matrix}
      1 ,& 1 \\ 3 ,& n+2
    \end{matrix}\, ;\, n \right) = \frac{2(n+1)}{n^2} \left( 1+ n - \frac{(n-1)!}{n^{n-1}}\left(e^n - \sum_{k=0}^{n-1}\frac{n^k}{k!} \right) \right).
  \]
  From Waadeland's tail theorem it follows that the continued fraction
  \(\contFracOpe_{m=1}^{\infty} \left(\frac{-n(m+n)}{m+2 n+2}\right)\) converges to a finite value \(f_1\) given by
  \[
    f_1 = -2(1+n)\left(1 - \frac{1}{\Sigma_\infty} \right).
  \]
  Clearly, \( \contFracOpe_{m=1}^{\infty} \left(\frac{-n(m+n-1)}{m+2 n + 1}\right)\) is also convergent and
  \begin{align*}
    \contFracOpe_{m=1}^{\infty} \left(\frac{-n(m+n-1)}{m+2 n + 1}\right)
    &= \cfrac{-n^2}{2(1+n) -  2(n+1)(1-\cfrac{1}{\Sigma_\infty})} = \frac{-n^2}{2(1+n)} \Sigma_\infty \\
    &=\frac{(n-1)!}{n^{n-1}}\left(e^n - \sum_{k=0}^{n-1}\frac{n^k}{k!} \right) -(1+ n).
  \end{align*}
  The result follows immediately.
\end{proof}

Next, we establish the explicit expression of the difference between successive convergents and the estimate of the
rate of convergence.

\begin{proposition}
  \label{prop:diff}
  Let \(k \in \mathbb{N}\) and \(C_k = P_k/Q_k \) be the convergent of the continued fraction expansion of \(e^n\) given in
  Theorem \ref{thm:exp_cont}. We have
  \[
    C_k - C_{k-1} = - \frac{n^{n+k+1}}{(n-1)! (n)_{k+1} (k+1)k}
  \]
  In particular, \(n\) divides the numerator of \(C_k - C_{k-1}\).
\end{proposition}

\begin{proof}
  From the Euler-Wallis identities, it holds that
  \begin{equation}
    \label{eqn:diff}
    \frac{P_{k}}{Q_{k}} - \frac{P_{k-1}}{Q_{k-1}} = \frac{(-1)^{k-1} a_1 a_2 \ldots a_{k}}{Q_{k}Q_{k-1}} \times \frac{n^{n-1}}{(n-1)!}
    = - \frac{n^{n+k-1} (n)_k }{(n-1)! Q_{k}Q_{k-1}},
  \end{equation}
  with \(a_k = - n ( k + n - 1) \). The denominators \(Q_k\) satisfy the recurrence relation \( Q_k = (k+2 n + 1)Q_{k-1} - n(k+n-1) Q_{k-2}\),
  with initial condition \(Q_{-1}=0, Q_0 =1 \), and it can easily be verified that a solution is given
  by \( Q_k = (1/n)(k+1)(n)_{k+1}\). Replacing this expression in (\ref{eqn:diff}) gives the result.
\end{proof}

The curious property of the convergents of (\ref{eqn:contfrac1}) mentioned in the introduction is just a special case
of the fact that in the continued fraction expansion of \(e^n\), the difference between the numerator of two successive
convergents is divisible by \(n\). The estimate on the rate of convergence follows immediately
from the proposition.

\begin{corollary}
  \label{cor:convspeed}
  Let \(C_k = P_k/Q_k \) be the convergents of the continued fraction expansion of \(e^n\) given in
  Theorem \ref{thm:exp_cont}. We have
  \[
    \left| e^n - C_k \right| = O\left(\frac{n^{k+1}}{(k+1)(k+2)(n)_{k+2}}\right). 
  \]
\end{corollary}

\begin{remark}
  The convergents
  \begin{align*}
    C_0 = 3, \quad C_1 = \frac{11}{4}, \quad C_3= \frac{87}{32},
  \end{align*}
  of \eqref{eqn:contfrac1} are also convergents of the regular continued fraction of $e$, that is, they are
  best Lagrange approximations (see e.g. \cite{K2008}). It may be interesting to investigate whether these are the only convergents
  of  \eqref{eqn:contfrac1} that are best Lagrange approximations of $e$. For instance, we verify numerically that among the first $200$
  convergents of the regular continued fraction of $e$ only the convergents $C_0$, $C_1$ and $C_3$ have this property. 
\end{remark}

Let us conclude this section by rewriting Theorem \ref{thm:exp_cont} in an equivalent form giving  the direct evaluation of the continued
fraction appearing in \eqref{eq:contfracen}.
  
\begin{corollary} \label{cor:directCont1}
  For $n\ge1$, we have
  \begin{equation} \label{eq:contFracGen1}
    1 + n + \contFracOpe_{m=1}^{\infty} \left(\frac{-n(m+n-1)}{m+2 n + 1}\right) = \frac{(n-1)!}{n^{n-1}}\left(e^n - \sum_{k=0}^{n-1}  \frac{n^k}{k!}\right).       
  \end{equation}
  
\end{corollary}

  The form given in Corollary \ref{cor:directCont1} is convenient for the generalization of $n$ to a complex variable, as we see in
  \S \ref{sec:ext}.
  
\begin{remark}
  Here, we point out the reason for using the shifted continued fraction in the proof
  of Theorem \ref{thm:exp_cont}. Using the original continued fraction gives a
  tail sequence from the associated recurrence relation (as in the proof of Proposition \ref{prop:diff}).
  However, the initial value of the tail sequence is \(t_0 = -b_1\), hence the sequence does not satisfy
  the hypothesis of Waadeland's theorem. In addition, a direct approach by finding a solution
  of the recurrence relation for the numerators \(P_k\) is more complicated.
\end{remark}

\section{Extension to complex variable}
\label{sec:ext}

The proof of Theorem \ref{thm:exp_cont} suggests that the continued fraction expansion of \(e^n\)
is just a special case of a more general result since Lemma \ref{lem:recurr} still holds when we replace \(n\) with \(z \in \mathbb{C}-\{0,-1,-2,\ldots\}\). This is indeed the case and in this section we discuss the generalization of Corollary \ref{cor:directCont1}.

\begin{theorem}
  \label{thm:incgamma}
  For \( z \in \{ z \in \mathbb{C} : |\arg(z)| < \pi  \} \), we have
  \begin{align}
    \label{eq:incgamma}
    \frac{\gamma(z,z)}{z^{z-1}e^{-z}} &= 1 + z + \contFracOpe_{m=1}^{\infty} \left(\frac{-z(m+z-1)}{m+2 z + 1}\right),
  \end{align}
  pointwise and uniformly in compacts.
\end{theorem}

\begin{proof}
  As mentioned above, it is clear that both lemmas and the proof of the theorem hold for
  \(z\) in the indicated cut plane, so it only remains to prove the uniform convergence in
  compacts. Suppose \(D\) is compact domain in the cut plane, next suppose
  \(M \in \mathbb{R}_{>0} \) is such that
  \[
    \left| \frac{z^2}{2(z+1)} \right| \leq M,
  \]
  for all \( z \in D\). Let \(f_k(z)\) denote the convergents of the continued fraction \eqref{eq:incgamma}
  and \( f(z) \) the pointwise limit of \(f_k(z) \), we have
  \[
    |f(z) - f_k(z)| \leq M |\Sigma_\infty(z) - \Sigma_{k-1}(z)|,
  \]
  where \(\Sigma_{k-1}(z)\) and \(\Sigma_\infty(z) \) are defined in terms of the shifted continued fraction as in the proof of
  Theorem \ref{thm:exp_cont}. The uniform convergence then follows from that of the hypergeometric
  series.
\end{proof}

\begin{corollary}
  \label{cor:confluent}
    For \( z \in \{ z \in \mathbb{C} : |\arg(z)| < \pi  \} \) and \( n \in \mathbb{N} \), it holds that
    \begin{align*}
      {}_1F_1\left(
    \begin{matrix}
      1 \\ z+1
    \end{matrix}\, ;\, z \right)&=  1 + z + \contFracOpe_{m=1}^{\infty} \left(\frac{-z(m+z-1)}{m+2 z + 1}\right)  \\
    \int_0^1 (1-t)^{n-1} e^{t n} d t &= \frac{1}{n}\left( 1 + n + \contFracOpe_{m=1}^{\infty} \left(\frac{-n(m+n-1)}{m+2 n + 1}\right) \right). \\
  \end{align*}
\end{corollary}

It would appear at first glance that the continued fraction expansion of Theorem \ref{thm:exp_cont}
could be extended to a continued fraction for arbitrary powers of \(e\). However, by the proof of
 Theorem \ref{thm:exp_cont} and from Theorem \ref{thm:incgamma} we see that the expansion of \(e^n \) is {\em accidental} in the
sense that it arises from special values of the incomplete gamma function \(\gamma(s,x)\).

Finally, it is interesting to compare the expansion in Corollary \ref{cor:confluent} to the
standard M-fraction representation
\begin{equation}
  \label{eq:confluenthyper1}
      {}_1F_1\left(
    \begin{matrix}
      1 \\ b+1
    \end{matrix}\, ;\, z \right) =  \frac{b}{b-z} \subplus \contFracOpe_{m=1}^{\infty} \left(\frac{m z}{b + m -z}\right),
\end{equation}
for \( z \in \mathbb{C} \) and \( b \in \mathbb{C} - \mathbb{Z}_{\leq 0}  \) (see e.g. (16.1.17) in \cite{Cuyt2008}). For instance, by setting \(b = z \in \{ z \in \mathbb{C} : |\arg(z)| < \pi  \}\) in \eqref{eq:confluenthyper1}  we obtain
\[
   {}_1F_1\left(
    \begin{matrix}
      1 \\ z+1
    \end{matrix}\, ;\, z \right) =  \frac{z}{0} \subplus \contFracOpe_{m=1}^{\infty} \left(\frac{m z}{m}\right),
\]
and is immediate to verify that the expansion is not related to the one in Corollary \ref{cor:confluent} by an equivalence transformation since the convergents are different.

Another known  continued fraction expansion of the
incomplete gamma function (see e.g. \cite{Lorentzen2008} page 280) is given by
\[
  \frac{ \gamma(a,z)}{z^a e^{-z}} = \frac{1}{a} \subminus \frac{a z}{1+a+z} \subplus
  \contFracOpe_{m=1}^{\infty} \left(\frac{-z(m+a )}{1+m+a+z}\right),
\]
valid for $a,z \in \mathbb{C}$. Setting $a=z$, we obtain the continued fraction expansion
\[
  \frac{ \gamma(z,z)}{z^{z-1} e^{-z}} = \frac{1}{1} \subminus \frac{z}{1+2 z} \subplus
  \contFracOpe_{m=1}^{\infty} \left(\frac{-z (m+z )}{1+m+2 z}\right),
\]
with a remarkably similar structure to the one given in Theorem \ref{thm:incgamma}. However, despite the
similarities, the two continued fractions expansions are not related by equivalence transformations or other
type of simple transformations (see also Remark \ref{rem:natural} for another form of \eqref{eq:incgamma} that
might help explain the difference between the two expansions).

\begin{remark}
  \label{rem:other}

  One of the main results of \cite{DZ2021} (Theorem 1), with the notation of the current paper, is the identity
  \begin{equation}
    \label{eq:DZresult}
    1 + k + \contFracOpe_{m=1}^{\infty} \left( \frac{a m}{k+ m + 1} \right) = \frac{a^{a+k+1}}{(a+k)! (e^a - \sum_{s=0}^{a+k} \frac{a^s}{s!})},
  \end{equation}
  for integers $a,k$ with $a+k\geq0$. The identity is also highlighted in \cite{Raayoni2019} as an example of a generalization of continued
  fractions conjectured by the ``Ramanujan machine'' algorithm.

  The continued fraction identity \eqref{eq:DZresult} can be obtained immediately by taking $z=a$ and $b = a + k +1$ in the standard
  M-fraction representation \eqref{eq:confluenthyper1} of the confluent hypergeometric function. Thus, by the remarks above on
  \eqref{eq:confluenthyper1}, it is clear that \eqref{eq:DZresult} does not result in continued fractions expansions of $e^n$ equivalent
  to the ones given in this paper.

\end{remark}

\section{Extension to rational exponent}
\label{sec:extension}

In the previous sections we noted that the limit values of the continued fractions in Theorem \ref{thm:exp_cont} depend essentially on
the special values of the generalized hypergeometric functions given in Lemma \ref{lem:hyper}. It is a natural to attempt
to extend Lemma \ref{lem:hyper} with the purpose of obtaining other continued fraction expansions. In this section, using this idea we
obtain a family of continued fractions for $e^{\frac{\ell}n}$ with \(1\leq\ell<n\). 

\begin{theorem} \label{thm:ratexp}
  Let \(1 \leq \ell < n \) be positive integers, then
  \begin{align}\label{eq:coneln}
    e^{\frac{\ell}{n}} = \sum_{k=0}^{\ell-1} \frac{\ell^k}{k! n^k} + \frac{\ell^{\ell-1}}{(n-1)n^{\ell-1}(\ell-1)!} \left[ 1 - \frac{n}{(n+\ell(n-1))n + (n-1)\contFracOpe_{m=1}^{\infty} \left( \frac{-\ell\, n (m-1+\ell)}{n(m+1+\ell) +\ell} \right)} \right]
  \end{align}
\end{theorem}

Note that, different from \eqref{eq:contfracen}, in \eqref{eq:coneln} the continued fraction, with partial numerators and
denominators with regular patterns, appears in the denominator of the right-hand side. Despite this apparent discrepancy,
the two families of continued fractions are actually related, as we show after the proof of the theorem.

Let us start by making some preparations for proof. As in the integral exponent case, we start by considering a continued fraction
\[
  \contFracOpe_{m=1}^{\infty} \left( \frac{- z (m+n z)}{m+ (n+1) z + 2} \right)
\]
with \(z,n \in \mathbb{C}\). The associated recurrence relation is then given by
\[
  X_k = (k + z (n+1) + 2) X_{k-1} - z( k+ n z) X_{k-2},
\]
with particular solution
\[
  X_k = \Gamma(k + 2 + n z) (k+2 + (n-1) z).
\]

As in Section \ref{sec:cont}, we define
\[
  t_j = - \frac{X_j}{X_{j-1}} = \frac{-(j+1+n z)(j+2+(n-1)z)}{(j+1+(n-1)z)},
\]
leading to
\[
  \Sigma_\infty =  {}_2F_2\left(
    \begin{matrix}
      (n-1)z +1  ,& 1 \\ n z +2 ,& (n-1)z+3
    \end{matrix}\, ;\, z \right).
\]
Next, we set \(z = \ell/n\) with \(1 \leq \ell < n \) and consider the corresponding special values of
the hypergeometric function above (c.f. Lemma \ref{lem:hyper}).

\begin{lemma} \label{lem:spv2}
  Let \(1 \le \ell < n \) be positive integers, then
  \begin{align*}
    &{}_2F_2\left(
    \begin{matrix}
      \frac{n-1}{n}\ell + 1  ,& 1 \\ \ell+2 ,& \frac{n-1}{n}\ell+3
    \end{matrix}\, ;\, \frac{\ell}{n} \right) \\
    & = (\ell+1)!\frac{ n^{\ell+1}}{\ell^{\ell+1}} \frac{\Gamma(\frac{n-1}{n}\ell+3)}{\Gamma(\frac{n-1}{n}\ell+1)} \left[ \frac{n}{\ell}\left( \sum_{k=0}^{\ell} \frac{\ell^k}{k! n^k} - e^{\frac{\ell}n}\right)
      + \frac{\ell^{\ell-1}}{(\ell-1)! n^{\ell-1}} \left( \frac{1}{n + \ell(n-1)} \right) \right].
  \end{align*}
\end{lemma}

\begin{proof}
  First, by Euler's integral transform, we have
  \begin{equation}
    \label{eq:fracp1}
    {}_2F_2\left(
      \begin{matrix}
        \frac{n-1}{n}\ell + 1  ,& 1 \\ \frac{n-1}{n}\ell+3  ,&  \ell+2
      \end{matrix}\, ;\, \frac{\ell}{n} \right) =
    \frac{\Gamma(\frac{n-1}{n}\ell+3)}{\Gamma(\frac{n-1}{n}\ell+1)} \int_{0}^1 t^{\frac{n-1}{n}\ell} (1-t) \, 
    {}_1F_1\left(
      \begin{matrix}
        1 \\ \ell +2
      \end{matrix}\, ;\, \frac{t \ell}{n} \right) d t.
  \end{equation}

  Moreover, by another application of the integral transform we see that
  \[
     {}_1F_1\left(
    \begin{matrix}
      1 \\ \ell +2
    \end{matrix}\, ;\, \frac{t \ell}{n} \right) = \frac{(\ell+1)n^{\ell+1}}{\ell^{\ell+1} t^{\ell+1}} e^{\frac{t \ell}{n}} \gamma\left(\ell+1, \frac{\ell t}{n}\right),
  \]
  and since
  \[
    \gamma\left(\ell+1, \frac{\ell t}{n}\right) = \ell! \left(1 - e^{-\frac{t \ell}{n}} \sum_{k=0}^{\ell} \frac{(t \ell)^k}{n^k k!}  \right),
  \]
  we obtain
  \[
    {}_1F_1\left(
    \begin{matrix}
      1 \\ \ell +2
    \end{matrix}\, ;\, \frac{t \ell}{n} \right) = \frac{(\ell+1)! n^{\ell+1}}{\ell^{\ell+1} t^{\ell+1}} \left( e^{\frac{t \ell}{n}} - \sum_{k=0}^{\ell} \frac{(t \ell)^k}{n^k k!}  \right).
  \]
  Moreover, by setting
  \[
    C_{n,\ell} := \frac{(\ell+1)! n^{\ell+1}}{\ell^{\ell+1}} \frac{\Gamma(\frac{n-1}{n}\ell+3)}{\Gamma(\frac{n-1}{n}\ell+1)},
  \]
  the integral in \eqref{eq:fracp1} is given by
  \[
    C_{n,\ell} \int_0^1 t^{-\frac{\ell}{n}-1} (1-t) \left( e^{\frac{t \ell}{n}} - \sum_{k=0}^{\ell} \frac{(t \ell)^k}{n^k k!}  \right) d t,
  \]
  by partial integration, this is equal to
  \begin{equation}
    \label{eq:eqlemma3}
    C_{n,\ell} \frac{n}{\ell} \int_0^1 t^{-\frac{\ell}n} \left( - \frac{1}{n}(e^{t\frac{\ell}n} (n + \ell(t-1)) + \frac{(\ell+1)\ell^{\ell} t^{\ell}}{n^\ell \ell!}  + \sum_{k=0}^{\ell-1} \frac{t^k \ell^k}{k! n^k} \left( (k+1) -\frac{\ell}{n}\right) \right) d t.
  \end{equation}
  Next, we notice that
  \[
    \int_0^1 t^{-\frac{\ell}n} e^{t\frac{\ell}{n}} (\ell(t-1) + n)d t = n e^{t\frac{\ell}{n}},
  \]
  since
  \[
    \ell \int_0^1 t^{-\frac{\ell}n+1} e^{t\frac{\ell}{n}} dt = n e^{\frac{\ell}n} - (n-\ell) \int_{0}^1 t^{-\frac{\ell}n} e^{t \frac{\ell}n} dt
  \]
  holds for \(1 \leq \ell < n \).
  Therefore, the integral \eqref{eq:eqlemma3} is equal to
  \[
    C_{n,\ell} \left[ \frac{n}{\ell} \left( \sum_{k=0}^\ell \frac{\ell^k}{k! n^k} - e^{\frac{\ell}{n}} \right) + \frac{\ell^\ell}{(\ell-1)!n^{\ell-1}}\left( \frac{1}{n + \ell(n-1)}\right)  \right],
  \]
  as desired.
\end{proof}

With these preparations, we give the proof of the main result of the section.

\begin{proof}[Proof of Theorem \ref{thm:ratexp}]

  Let us consider the continued fraction
  \[
    \contFracOpe_{m=1}^{\infty} \left( \frac{- z (m-1+n z)}{m+ (n+1) z + 1} \right),
  \]
  then, by setting \(z = \frac{\ell}n\) with \(1 \le \ell < n \), as before, and using an equivalent transformation
  and the Waadeland tail theorem (cf. proof of Theorem \ref{thm:exp_cont}), along with the fact
  \[
    \frac{(\ell+1)! n^{\ell+1}}{\ell^{\ell+1}} \frac{\Gamma(\frac{n-1}{n}\ell+3)}{\Gamma(\frac{n-1}{n}\ell+1)} = \frac{(\ell+1)! n^{\ell-1}(n + \ell(n-1))(n-1)(\ell+1)}{\ell^{\ell+1}},
  \]
  we obtain
  \begin{align*}
    \frac{1}{n} \contFracOpe_{m=1}^{\infty}& \left( \frac{-\ell\, n (m-1+\ell)}{n(m+1) + (n+1)\ell} \right) = -\frac{n + \ell(n-1)}{n-1} \\
                                      &\qquad - \frac{\ell^{\ell-1}n}{ (\ell-1)! (n-1)^2 n^\ell (e^{\frac{\ell}n} - \sum_{k=0}^\ell \frac{\ell^k}{k! n^k} ) - \ell^{\ell-1}(n-1)},
  \end{align*}
and the result is obtained by factoring $e^{\frac{\ell}n}$.
\end{proof}

As mentioned above, Theorem \ref{thm:exp_cont} may be considered as a limit case of Theorem \ref{thm:ratexp}. Indeed, the identity \eqref{eq:coneln} is equivalent to
  \begin{align} \label{eq:elnrewriten}
    e^{\frac{\ell}{n}} = \sum_{k=0}^{\ell-1} \frac{\ell^k}{k! n^k} + \frac{\ell^{\ell-1}}{n^{\ell-1}(\ell-1)!} \left[\frac{n(\ell+1) + n\contFracOpe_{m=1}^{\infty} \left( \frac{-\ell\, n (m-1+\ell)}{n(m+1+\ell) +\ell} \right) }{(n+\ell(n-1))n + (n-1)\contFracOpe_{m=1}^{\infty} \left( \frac{-\ell\, n (m-1+\ell)}{n(m+1+\ell) +\ell} \right)} \right],
  \end{align}
  and by setting $n=1$ we recover the result of Theorem \ref{thm:exp_cont}. We note that taking $n=1$ in \eqref{eq:elnrewriten} is allowed, since in this case Lemma \ref{lem:spv2} reduces to Lemma \ref{lem:hyper}. The convergence of other cases with $1 \neq n \leq \ell$ need to be considered by separate and we do not further investigate this issue in this paper.

  The following result, giving the explicit limit value of two-parameter families of continued fractions is obtained immediately from Theorem \ref{thm:ratexp}.

\begin{corollary} \label{cor:cfvalue2}
  For $1 \leq \ell< n$,  or $n=1$ (cf. \eqref{eq:contFracExpN}), we have
  \[
    n(\ell+1) + \ell + \contFracOpe_{m=1}^{\infty} \left( \frac{-\, n \ell( m -1 + \ell)}{n (m + \ell+1) +\ell} \right) =
    \frac{\ell^\ell + (n-\ell(n-1))(\ell-1)! n^{\ell-1} \left( e^{\frac{\ell}{n}} - \sum_{k=0}^{\ell-1} \frac{\ell^k}{k! n^k} \right) }{\ell^{\ell-1} - (n-1)(\ell-1)! n^{\ell-1} \left( e^{\frac{\ell}{n}} - \sum_{k=0}^{\ell-1} \frac{\ell^k}{k! n^k} \right)}
  \]
\end{corollary}

We remark that all the coefficients in the right-hand side of the identity of the corollary above are integers.

\begin{example}
  Corollary \ref{cor:cfvalue2} has a particular simple expression for the case \(\ell=1\). Namely,
  \[
    1+ 2n + \contFracOpe_{m=1}^{\infty} \left( \frac{-\, n m}{1 + n (m + 2)} \right) = \frac{e^{\frac{1}{n}}}{n-(n-1)e^{\frac{1}{n}}}.
  \]

  From the point of view of the patterns on the partial numerator and denominators in the continued fraction, this family may be
  considered as the natural generalization of \eqref{eqn:contfrac1}. Indeed, as we mentioned in the introduction, the case $n=1$ is equal
  to \eqref{eqn:contfrac1}, and, for instance, we have
  \begin{align*}
    \frac{e^{\frac{1}{2}}}{2-e^{\frac{1}{2}}} &= 5 - \frac{2}{7} \subminus \frac{4}{9} \subminus \frac{6}{11} \subminus \frac{8}{13} \subminus \subcdots, \\
    \frac{e^{\frac{1}{3}}}{3-2e^{\frac{1}{3}}} &= 7 - \frac{3}{10} \subminus \frac{6}{13} \subminus \frac{9}{16} \subminus \frac{12}{19} \subminus \subcdots
  \end{align*}
  and similar ones for $n\geq4$.
\end{example}

\begin{example}
  Let us give also the explicit identities of Corollary \ref{cor:cfvalue2}  for the cases $\ell=2,3$. Concretely, we have
  \[
    2 + 3 n + \contFracOpe_{m=1}^{\infty} \left( \frac{-\, n( m +1)}{2+ n (m + 3)} \right) = \frac{n+(2-n)e^{\frac{2}{n}}}{n+1 - (n-1)e^{\frac{2}{n}}},
  \]
  valid for $n>2$, and
  \[
    3+ 4 n + \contFracOpe_{m=1}^{\infty} \left( \frac{-\, n( m + 2)}{3+ n (m +4)} \right) =  \frac{2n(3+2n)+2n(3-2n)e^{\frac{3}{n}}}{3+4n+2n^2 - 2n(n-1)e^{\frac{3}{n}}}
  \]
  for $n>3$.
\end{example}

It is immediate to verify that the families of continued fractions given in this section are not equivalent to the standard M-fraction representations of the confluent hypergeometric function.

We leave the discussion on the convergence rate and analytic properties of the families of Theorem \ref{thm:ratexp} and Corollary \ref{cor:cfvalue2} for another occasion.

\begin{remark}\label{rem:natural}
  The case $n=1$ in Corollary \ref{cor:cfvalue2} gives
  \begin{equation}\label{eq:contFracExpN}
    1 + 2n + \contFracOpe_{m=1}^{\infty} \left(\frac{-n(m+n-1)}{m+2 n + 1}\right) = \frac{(n-1)!\left(e^n - \sum_{k=0}^{n-1}  \frac{n^k}{k!} \right) + n^n}{n^{n-1}}.
  \end{equation}
  which is an equivalent form of Corollary \eqref{eq:contFracExpN}. Similarly, it may be more natural to consider
  \begin{align*}
    1 + 2z + \contFracOpe_{m=1}^{\infty} \left(\frac{-z(m+z-1)}{m+2 z + 1}\right) =  \frac{\gamma(z,z)}{z^{z-1}e^{-z}} + z,
  \end{align*}
  in Theorem \ref{thm:incgamma}.
\end{remark}

\section*{Acknowledgments}

The work was supported by JST CREST, Japan [Grant Number JPMJCR14D6]. The author is grateful to Masato Wakayama
and Hiroyuki Ochiai for comments and discussion regarding this research. In addition, the author is grateful to
J.H. McCabe and Gerardo González Robert for comments on preliminary versions of the article.

\vspace{1em}

\begin{flushleft}
  Cid Reyes-Bustos \par
  Department of Mathematical and Computing Science, School of Computing, \par
  Tokyo Institute of Technology \par
  2 Chome-12-1 Ookayama, Meguro, Tokyo 152-8552 JAPAN \par\par
  \texttt{reyes@c.titech.ac.jp}
\end{flushleft}


\begin{thebibliography}{99}

\bibitem{AAR1999}
  G.~E.~Andrews, R.~Askey and R.~Roy:
  Special functions,
  Encyclopedia of Mathematics and its Applications, Cambridge University Press, 1999.

\bibitem{Cuyt2008}
  A.~Cuyt, V.~Brevik~Petersen, B.~Verdonk, H.~Waadeland and W.~B.~Jones:
  Handbook of Continued Fractions for Special Functions,
  Springer Science+Business Meida B.V., 2008.

\bibitem{DZ2021}
  R.~Dougherty-Bliss and  D.~Zeilberger:
  {\textit Automatic conjecturing and proving of exact values of some infinite families of infinite continued fractions.}
  Ramanujan J. (2021). Published Online 24 February 2021. \texttt{https://doi.org/10.1007/s11139-020-00345-z}.
  
\bibitem{KRW2017}
  K.~Kimoto, C.~Reyes-Bustos and M.~Wakayama:
  \textit{Determinant expressions of constraint polynomials and degeneracies of the asymmetric quantum Rabi model}.
  International Mathematics Research Notices, 2020 (Published Online April 20, 2020). \texttt{https://doi.org/10.1093/imrn/rnaa034}.

\bibitem{K2008}
  S.~Khrushchev:
  Orthogonal Polynomials and Continued Fractions, From Euler's Point of View,
  Encyclopedia of Mathematics and its Applications 122, Cambridge University Press, 2008.

        
 \bibitem{Lorentzen2008}
  L.~Lorentzen~and~H.~Waadeland:
  Continued Fractions,
  Atlantis Studies in Mathematics for Engineering and Science 1,
  Atlantic Press/World Scientific 2008.

\bibitem{McCabe1983}
  J.~H.~McCabe:
  \textit{The Quotient-Difference Algorithm and the Padé Table: An Alternative Form and a General
    Continued Fraction}, Math. Comp. {\bf 41}, 163, (1983). 183-197.

\bibitem{McCabe2009}
  J.~H.~McCabe:
  \textit{On the Padé table for $e^x$ and the simple continued fractions for $e$ and $e^{L/M}$},
  Math. Comp. {\bf 41}, 163, (1983). 183-197.
  
\bibitem{Raayoni2019}
  G.~Raayoni, S.~Gottlieb, Y.~Manor, G.~Pisha, Y.~Harris, U.~Mendlovic, D.~Haviv, Y.~Hadad and I.~Kaminer:
  \textit{Generating conjectures on fundamental constants with the Ramanujan Machine},
  Nature {\bf 590}, 67–73 (2021). \texttt{https://doi.org/10.1038/s41586-021-03229-4}.

\bibitem{RB2018}
  C.~Reyes-Bustos:
  \textit{Study on the spectrum of the asymmetric quantum Rabi model},
  PhD Thesis, Kyushu University, (2018). 
  \texttt{https://catalog.lib.kyushu-u.ac.jp/opac\_download\_md/1959076/math0234.pdf}  

\bibitem{CRB2020}
 C.~Reyes-Bustos,
 \textit{Extended divisibility relations for constraint polynomials of the asymmetric quantum Rabi model},
 in ``International Symposium on Mathematics, Quantum Theory, and Cryptography (MQC 2019)'', eds. T. Takagi et al.
 Mathematics for Industry \textbf{33}, 149-168, Springer Singapore, 2020.
  
\bibitem{Whittaker1950}
  E.~T.~Whittaker~and~G.~N.~Watson:
  A Course of Modern Analysis,
  Fourth edition, Cambridge Mathematical Library,
  Cambridge University Press, 1950.

\bibitem{Lu2019}
  Z.~Lu:
  \textit{Elementary proofs of generalized continued fraction formulae for $e$},
  arXiv:1907.05563, 2019.

\end{thebibliography}
\end{document}